\documentclass{amsart}

\usepackage{amssymb,amscd}
\usepackage{amsmath,latexsym,amssymb, amscd}
\usepackage{amssymb,latexsym, xspace, enumerate}

\usepackage{graphicx}

\newtheorem{theorem}{Theorem}[section]

\long\def\alert#1{\smallskip{\hskip\parindent\vrule%
\vbox{\advance\hsize-2\parindent\hrule\smallskip\parindent.4\parindent%
\narrower\noindent#1\smallskip\hrule}\vrule\hfill}\smallskip}

\theoremstyle{definition}
\newtheorem{definition}[theorem]{Definition}

\newtheorem{example}[theorem]{Example}

\newtheorem{discussion}[theorem]{Discussion}
\newtheorem{remark}[theorem]{Remark}

\newtheorem{questions}[theorem]{Questions}

\numberwithin{equation}{section}

\newtheorem{remark/questions}[theorem]{Remark and Questions}

\newtheorem{remarks}[theorem]{Remarks}

 \long\def\alert#1{\smallskip\line{\hskip\parindent\vrule
\vbox{\advance\hsize-2\parindent\hrule\smallskip\parindent.4\parindent
  \narrower\noindent#1\smallskip\hrule}\vrule\hfill}\smallskip}

\def\frac#1#2{{{#1}\over{#2}}}

\def\endresult#1{\medskip}

\def \Q{{{\mathbb Q }}}
\def \Z{{{\mathbb Z }}}
  \def \N{{\mathbb  N}}

\def\q{{\bf q}}
\def\p{{\bf p}}
\def\m{{\bf m}}
\def\n{{\bf n}}

\def\Gff{\mathop{\rm Gff}}

\def\dim{\mathop{\rm dim}}

\def\Spec{\mathop{\rm Spec}}

\def\hgt{\mathop{\rm ht}}

\begin{document}
\baselineskip 15 pt

\title[Formal fibers ]{
 Formal fibers of prime ideals in polynomial rings}
\author{William Heinzer}
\address{Department of Mathematics, Purdue University, West
Lafayette, Indiana 47907}

\email{heinzer@math.purdue.edu}

\author{Christel Rotthaus}
\address{Department of Mathematics, Michigan State University, East
  Lansing,
   MI 48824-1027}

\email{rotthaus@math.msu.edu}

\author{Sylvia Wiegand}
\address{Department of Mathematics, University of Nebraska, Lincoln,
NE 68588-0130}

\email{swiegand@math.unl.edu}




   \thanks{The authors are grateful for
  the hospitality, cooperation and support of Michigan State,
  Nebraska,
   Purdue, and MSRI in Berkeley,  where we worked on 
this research.}

\thanks{ We thank the referee for helpful suggestions about this paper.}

   \begin{abstract} Let $(R,\m)$ be a Noetherian local domain of dimension 
$n$ that is essentially finitely generated over a field  and
    let $\widehat R$ be the $\m$-adic completion of $R$. 
Matsumura has shown that $n-1$ is the maximal height possible for prime 
ideals $P$ of $\widehat R$ such that $P\cap R=(0)$.
In this article we prove that $\hgt {P}=n-1$, for {\it every}  prime ideal 
$ P$ of $\widehat R$ that is maximal with respect to $P \cap R = (0)$.
We also present a  related result  concerning  generic  formal fibers
of certain extensions of mixed polynomial-power series rings.
\end{abstract}

\maketitle

\section{Introduction}\label{intro}

Let $(R,\m)$ be a Noetherian local domain and let $\widehat R$ be the 
$\m$-adic completion of $R$. The {\it generic formal fiber ring} of $R$ is
the localization $(R \setminus (0))^{-1}\widehat R$ of $\widehat R$ with respect
to the multiplicatively closed set of nonzero elements of $R$.  Let $\Gff(R)$ denote
the generic formal fiber ring of $R$.  If
$R$  
is essentially finitely generated over a field and $\dim R = n$, then  
$\dim( \Gff(R))  = n-1$ by the result of
Matsumura  \cite[Theorem~2]{M1.5} mentioned in the abstract.  
In this article  we show {\it every} maximal 
ideal of $\Gff(R)$ has height $n-1$; equivalently,  
$\hgt P = n-1$, 
for every prime ideal 
$ P$ of $\widehat R$ that is maximal with respect to $P \cap R = (0)$,  a sharpening of Matsumura's result.   

In several earlier articles we encounter formal fibers and generic fibers; these concepts are
related to  our study of prime 
spectral maps among ``mixed polynomial-power series rings" over a field; see 
 \cite{weier}, \cite{pps} and \cite{tgf}. 
For $P \in \Spec R$, the {\it formal fiber} over $P$ is  
$\Spec ((R\setminus P)^{-1}(\widehat R/P\widehat R))$, or equivalently 
$\Spec((R_P/PR_P)\otimes_R \widehat R)$.
Let $\Gff(R/P)$ denote 
the generic formal fiber 
ring  of $R/P$. 
Since $\widehat R/P\widehat R$ is 
the completion of $R/P$, the formal fiber over $P$ is   $\Spec (\Gff(R/P))$.

Let $n$ be a positive integer, let $X=\{x_1,\hdots,x_n\}$ be a set
of $n$ variables over   a field $k$, and let
$A:=k[x_1,\hdots,x_n]_{(x_1,\hdots,x_n)} = k[X]_{(X)}$ denote the
localized polynomial ring in these $n$ variables over $k$.  Then
the completion of $A$ is $\widehat{k[X]}=k[[X]]$. 

With this notation, we prove   the following theorem  in \cite{weier}:

\begin{theorem}  \label{weierkxt} $\phantom{x}$ \cite[Theorem~24.1.1]{powerbook} Let $A = k[X]_{(X)}$ be the 
localized  polynomial ring as defined above. Every maximal ideal of
the generic formal fiber ring $\Gff(A)$  has height $n-1$. Equivalently, if $ Q$ is an ideal of
 $\widehat A$ maximal with  respect to $Q\cap A=(0)$, then $Q$ is a prime
ideal of height $n-1$.
\end{theorem}

We were inspired to revisit and generalize  Theorem~\ref{weierkxt}
by Youngsu Kim.    His interest 
in formal fibers and the material in \cite{weier}
 inspired us to consider  the  second question below. 

\begin{questions}\label{Ysk} For $n \in \N$, let $x_1, \ldots, x_n$ be indeterminates 
over a field $k$  and let 
$R = 
k[x_1,\dots,x_n]_{(x_1, \dots,x_n)}$ denote the localized polynomial ring 
with maximal ideal   $\m = (x_1,\dots,x_n)R$. Let 
$\widehat{R}$ be the $\m$-adic completion of $R$. 
\begin{enumerate} 
\item For $P \in \Spec R$, what is the dimension of the generic formal fiber  ring  
$\Gff (R/P)$?
\item  What heights are possible for maximal ideals of the ring $\Gff(R/P)$?
\end{enumerate}
\end{questions}

In connection with Question~\ref{Ysk}.1, for $P \in \Spec R$, the ring $R/P$ is
essentially finitely generated over a field, and a result of Matumura
\cite[Corollary, p.~263]{M1.5}
states  that $\dim (\Gff (R/P)) = n-1 - \hgt P$.

As a sharpening of Matsumura's result  and of  
Theorem~\ref{weierkxt}, we prove  Theorem~\ref{esfingen1};  see 
also Theorem~\ref{esfingen2}.  Thus the answer to  
Question~\ref{Ysk}.2 is  that the height of {\it every} maximal 
ideal of $\Gff(R/P)$ is $n-1 - \hgt P$.

\begin{theorem} \label{esfingen1} Let $S$ be a local domain essentially finitely generated
over a field;  thus  $S = k[s_1, \ldots, s_r]_{\p}$, where $k$ is a field, $r \in \N$,
 the elements
$s_i$ are in $S$ and $\p$ is a prime ideal of the  finitely generated $k$-algebra $k[s_1, \ldots, s_r]$. 
Let $\n := \p S$ and let $\widehat S$ denote the $\n$-adic completion of $S$. 
Then every maximal ideal of $\Gff (S)$  has height $\dim S - 1$.
Equivalently,  if $Q \in \Spec \widehat S$ is maximal with respect to $Q \cap S = (0)$, 
then $\hgt Q = \dim S - 1$.
\end{theorem}

In Theorem~\ref{wogffgen}  of  Section~\ref{other}, we prove 
that all 
maximal ideals in the generic formal fiber of certain extensions
of a mixed polynomial-power series ring have the same height.

\section{Background   and Preliminaries}

We begin with historical remarks concerning dimensions and heights of  maximal ideals  of generic formal fiber rings for 
Noetherian local domains: 

\begin{remarks}\label{gffrem} 
(1)  Let $(R,\m)$ be an $n$-dimensional  Noetherian local domain. 
In his paper \cite{M1.5}, Matsumura remarks that as the ring  $R$ gets closer to its
$\m$-adic completion $\widehat R$, it is natural to think that the dimension of 
the generic formal fiber ring $\Gff(R)$ gets 
smaller. Matsumura describes examples where $\dim (\Gff(R))$ has one of
the three values $n-1, n-2$ or $0$, and speculates \cite[p.261]{M1.5} as to
whether these are the only possible values for $\dim (\Gff(R))$.

(2)  Matsumura's question in item~1  is answered by Rotthaus    
in \cite{R3.5};  she  establishes the following result: 
For every positive integer $n$ and every integer $t$ between $0$ and $n - 1$, 
there exists an excellent regular local ring
$R$ such that $\dim R = n$ and such that the generic formal fiber
ring of $R$ has dimension $t$.

(3) 
For $(R,\m)$ an $n$-dimensional  universally catenary Noetherian local domain, Loepp and Rotthaus in
\cite{LR} compare the dimension of the generic formal fiber ring of $R$ with that
of the localized polynomial ring $R[x]_{(\m,x)}$.    Matsumura shows in \cite{M1.5}
that the dimension of the generic formal fiber ring $\Gff(R[x]_{(\m, x)})$ is either $n$
or $n-1$.   Loepp and Rotthaus in \cite[Theorem~2]{LR} prove that  $\dim (\Gff(R[x]_{(\m, x)})) = n$ 
implies that $\dim( \Gff R) = n-1$.  They  show by example that in general the converse is
not true, and they  give sufficient conditions for the converse to hold.

(4)  Let $(T,M)$ be a complete Noetherian local domain that contains a field of
characteristic zero.  Assume that   $T/M$ has cardinality at least 
the cardinality of the real numbers. 
 In the
articles \cite{Lo} and \cite{Lo2}, Loepp adapts techniques developed by Heitmann in \cite{H2} and proves, among other things,  
  for every prime ideal $p$ of $T$ with 
$p \ne M$,
there exists an excellent regular local ring $R$ that has completion $T$ and
has generic formal fiber ring $\Gff(R) = T_p$. By varying the height of $p$,
Loepp obtains examples where the dimension of the generic formal fiber ring
is any integer $t$ with $0 \le t < \dim T$. Loepp  shows for these
examples that there exists a unique prime
$q$ of $T$ with $q \cap R = P$ and $q = PT$,  for each nonzero prime $P$ of $R$.

(5) If $R$ is an $n$-dimensional  countable Noetherian local domain, Heinzer, Rotthaus and Sally show
in  \cite[Proposition 4.10, page 36]{HRS}  that: 
\begin{enumerate}
 \item[(a)] The generic formal fiber ring $\Gff(R)$ is a Jacobson  ring in 
the sense that each prime ideal of $\Gff(R)$ is 
an intersection of maximal ideals of $\Gff(R)$.
 \item[(b)] 
$\dim (\widehat R/P) = 1$ for
each prime
ideal $P \in \Spec \widehat R$ that is maximal with respect to 
$P \cap R = (0)$. 
\item[(c)]   If $\widehat R$ is equidimensional, then $\hgt P = n-1$ for 
each prime
ideal $P \in \Spec \widehat R$ that is maximal with respect to 
$P \cap R = (0)$. 
\item[(d)]  If $Q \in \Spec \widehat R$ with $\hgt Q \ge 1$, then there exists 
a prime ideal $P \subset Q$ such that $P \cap R = (0)$ and $\hgt (Q/P) = 1$. 
\end{enumerate} 
 It
follows from this result  that all ideals maximal 
in the generic formal fiber ring of the ring $A$ of Theorem~\ref{weierkxt} have the same height, if the field $k$ is countable.

(6) In Matsumura's article  \cite{M1.5} referenced in item 1 above, 
he does not address the question of whether
all ideals maximal in the generic formal fiber rings have the same height for the rings he studies. 
In general,  for an excellent regular local ring R, it can happen that 
\index{RLR}
 $\Gff(R)$  contains maximal ideals of different heights; see  the article \cite[Corollary 3.2]{R3.5} of Rotthaus.
\index{Rotthaus}

(7) Charters and Loepp in \cite[Theorem~3.1]{CL} extend Rotthaus's result of item~6:  Let 
$(T,M)$ be a complete Noetherian local ring and let $G$
be a nonempty subset of $\Spec T$ such that the number of maximal 
elements of $G$ is finite. 
They
  prove  there exists a 
Noetherian local domain $R$ whose completion is $T$  and whose
generic formal fiber is exactly $G$ if  $G$ satisfies the
following conditions: 
\begin{enumerate}
\item[(a)] $M \notin G$ and $G$ contains the associated primes of $T$,
\item[(b)]  If $P \subset Q$ are in $\Spec T$  and $Q \in G$, then $P \in G$, and
\item[(c)]  Every $Q \in G$ meets the prime subring of $T$ in $(0)$.
\end{enumerate} 
If $T$ contains the ring of integers and, in addition to
items~a, b, and c, one also has
\begin{enumerate}
\item[(d)] $T$ is equidimensional, and
\item[(e)] $T_P$ is a regular local ring for each maximal element $P$ of $G$,
\end{enumerate}
then Charters and Loepp  prove there
exists an excellent local domain $R$ whose completion is $T$ and whose
generic formal fiber is exactly $G$; see  \cite[Theorem~4.1]{CL}.
Since the maximal elements of the set $G$ may be chosen to have  different heights, this result
provides many examples where the generic formal fiber ring
contains maximal ideals of different heights.
\end{remarks}

We make  the following observations concerning injective local maps of Noetherian local rings:

\begin{discussion} \label{subspace1}
Let $\phi : (R,\m) \hookrightarrow (S,\n)$  be an injective local map 
of the Noetherian local ring
$(R,\m)$ into a Noetherian local ring $(S,\n)$.  Let
$\widehat R = \underset{n}{\varprojlim}\phantom{x} R/\m^n$ denote the $\m$-adic completion of $R$
and let $\widehat S = \underset{n}{\varprojlim}\phantom{x} S/\n^n$ denote the $\n$-adic completion
of $S$.
 For each $n \in \N$, we have $\m^n \subseteq \n^n \cap R$. Hence  there
exists a map  
$$
\phi_n : R/\m^n ~\to  ~R/(\n^n \cap R) ~ \hookrightarrow  ~S/\n^n, \quad  \text{ for each } ~n \in  \N.
$$
The family of maps $\{ \phi_n \}_{n \in \N}$ determines  a  unique map 
$\widehat \phi : \widehat R \to \widehat S$.

Since 
$\m^n \subseteq \n^n \cap R$,  
the $\m$-adic topology on $R$ is the  subspace topology   from $S$    
if and only if for each positive integer $n$
there exists a positive integer $s_n$ such that $\n^{s_n} \cap R \subseteq \m^n$.
Since $R/\m^n$ is Artinian, the descending chain of ideals $\{\m^n + (\n^s \cap R) \}_{s \in \N}$
stabilizes.  The ideal $\m^n$ is closed in the $\m$-adic topology, and it is  
closed in the subspace topology if and only if
$\bigcap_{s \in \N} (\m^n + (\n^s \cap R)) = \m^n$.
Hence  $\m^n$ is closed in the subspace topology if and only if there exists a positive
integer $s_n$ such that $\n^{s_n} \cap R \subseteq \m^n$. 
Thus the subspace topology  from $S$ is the same as the  $\m$-adic topology on $R$
if and only if $\widehat \varphi$ is injective.
\end{discussion}

\section{$\Gff(R)$ and $\Gff(S)$ for $S$ an extension domain of $R$\label{main}}

We use the following definition in our main result.

\begin{definition} \label{TGFdef}   Let  $R$ and $S$ 
be integral domains with
$R$ a subring of $S$ and 
let   $R\overset{\varphi}\hookrightarrow S$ denote the inclusion map of $R$ into $S$. 
The integral domain  $S$ is a {\it trivial generic
fiber} 
extension of $R$, or a {\it TGF} extension of $R$,  if
every nonzero prime ideal of $S$ has nonzero intersection with
$R$. In this case, we also say that $\varphi$ is 
  a {\it trivial generic fiber extension} or
{\it TGF extension}.\end{definition}

Theorem~\ref{tgflem} is useful in considering properties of generic formal fiber rings.

\begin{theorem} \label{tgflem}
Let $\phi : (R,\m) \hookrightarrow (S,\n)$ be an injective local map of 
Noetherian local integral domains. Consider the following properties$~\!\!:$ 
\begin{enumerate}
\item[$(1)$]  $\m S$ is $\n$-primary,  and $S/\n$ is finite algebraic over $R/\m$. 
\item[$(2)$]   $R \hookrightarrow S$ is a TGF-extension and $\dim R = \dim S$.  
\item[$(3)$]  $R$ is analytically irreducible.
\item[$(4)$]    $R$ is analytically normal and   $S$ is universally catenary.  
\item[$(5)$]  All maximal ideals of $\Gff (R)$ have the same height.
\end{enumerate}
If items 1, 2 and 3 hold, then $\dim  (\Gff (R)) = \dim (\Gff (S))$.    
If, in addition, items 4 and 5 hold,  
then the height $h$ of every maximal ideal of $\Gff (S)$ is  $h = \dim ( \Gff(R))$. 
\end{theorem} 

\begin{proof}
Let $\widehat R$ and $\widehat S$ denote the $\m$-adic completion of $R$ and $\n$-adic
completion of~$S$, respectively, and let $\widehat \phi: \widehat R \to \widehat S$ 
be the natural extension of $\varphi$ as given in Discussion~\ref{subspace1}. Consider the commutative diagram 
\begin{equation*} 
\CD
\widehat R @>{\widehat \phi}>> \widehat S~\phantom{,}\\
@AAA @AAA ~\phantom{,} \\
R @>{\phi}>> S~,
\endCD\tag{\ref{tgflem}.a}
\end{equation*}
where the vertical maps are the natural inclusion maps to the completion. 
Assume items~1, 2 and 3 hold. 
Item~1 implies that   $\widehat S$ is a finite $\widehat R$-module
with respect to the map $\widehat \phi$ by   \cite[Theorem~8.4]{M}.   
By item~2, we  have $\dim \widehat R = \dim R = \dim S 
= \dim \widehat S$.     Item~3  says that $\widehat R$ is an integral domain.
  It follows that the map $\widehat \phi :\widehat R \hookrightarrow \widehat S$ is injective.
Let $Q \in \Spec \widehat S$ and let $P = Q \cap \widehat R$. Since $R \hookrightarrow S$ 
is a TGF-extension, by item~2,  commutativity of Diagram~\ref{tgflem}.a implies that 
$$
Q ~ \cap ~S~ =~ (0)~ \iff ~ P~ \cap ~ R ~ = ~ (0).
$$ 
Therefore $\widehat \phi$ induces an injective
finite map $\Gff(R) \hookrightarrow  \Gff(S)$. We  conclude that  $\dim (\Gff(R)) = \dim (\Gff(S))$.

Assume in addition  that items~4 and 5 hold, and let   $h = \dim (\Gff(R))$.  
The assumption that $S$ is universally catenary implies 
that  $\dim (\widehat S/\q)   = \dim S$  for each 
minimal prime $\q$ of $\widehat S$ by \cite[Theorem~31.7]{M}.   
Since $\frac{\widehat R}{\q \cap \widehat R} \hookrightarrow \frac{\widehat S}{\q}$  
is an integral extension, we have $\q \cap \widehat R = (0)$.  
The assumption that $\widehat R$ is a normal domain implies
that   the going-down  theorem holds for 
$\widehat R \hookrightarrow \widehat S/\q$  by \cite[Theorem~9.4(ii)]{M}.  
Therefore for each $Q \in \Spec \widehat S$ we have $\hgt Q = \hgt P$, 
where $P = Q \cap \widehat R$.   Hence if
$\hgt P = h$ for each $P \in \Spec \widehat R$ that is maximal with 
respect to $P \cap R = (0)$, then $\hgt Q = h$ for each $Q \in \Spec \widehat S$
that is maximal with respect to $Q \cap S = (0)$.   
This
completes the proof of Theorem~\ref{tgflem}.
\end{proof}

\begin{remark}\label{essfgnfg} We would like to thank Rodney Sharp and Roger Wiegand 
for their interest in Theorem~\ref{tgflem}. 
The hypotheses  of Theorem~\ref{tgflem} do not necessarily imply 
that $S$ is a finite $R$-module,
or even  that $S$ is essentially finitely generated over $R$.   
If  $\phi : (R,\m) \hookrightarrow (T,\n)$
 is an extension of rank one discrete valuation rings (DVR's)  such that  $T/\n$ is finite algebraic over $R/\m$,   then  for every field $F$  that
 contains $R$  and is contained in 
 the field of fractions of $\widehat T$,  the ring $S := \widehat T \cap F$ is a DVR such that   
the extension $R\hookrightarrow S$ 
satisfies the hypotheses of  Theorem~\ref{tgflem}.   

As a specific example where $S$ is essentially finite over $R$, but not a finite $R$-module,  
 let $R=\Z_{5\Z}$, the integers localized at the prime ideal generated by $5$, and let $A$ be the integral closure of $\Z_{5\Z}$ in $\Q[i]$.
Then $A$ has two maximal ideals lying over $5R$, namely $(1+2i)A$ and $(1-2i)A$. Let $S=A_{(1+2i)A}$. 
Then the extension $R\hookrightarrow S$ 
satisfies the hypotheses of  Theorem~\ref{tgflem}. Since $S$ properly contains $A$, and every 
element in the field of fractions of $A$ that is integral over $R$ is contained in $A$,
it follows that $S$ is not finitely generated as an $R$-module.   In Remark~\ref{essfgnfg2},  we describe examples in higher dimension where $S$ is not a
finite $R$-module.
\end{remark}

\begin{discussion} \label{discuss1}
As in the statement of Theorem~\ref{esfingen1},  let $S = k[z_1, \ldots, z_r]_{\p}$   be  a local domain essentially 
finitely generated over a field $k$.   We observe that  $S$ is a localization at a maximal ideal of an 
integral domain that is  a finitely generated 
algebra over an extension  field $F$ of $k$. 

To see this, let $A = k[x_1, \ldots, x_r]$ be a polynomial ring in $r$ variables  over $k$, and let $Q$ denote the
kernel of the $k$-algebra homomorphism of  $A$ onto $k[z_1, \ldots, z_r]$  defined by mapping $x_i \mapsto z_i$
for each $i$ with $1 \le i \le r$.  
Using permutability of localization and residue class formation,   there exists 
a prime ideal $N \supset Q$ of $A$
such that  $S = A_N/QA_N$.
 A version of 
Noether normalization  as in \cite[Theorem 24  (14.F) page 89]{M1} states that, if $\hgt N = s$, 
then 
there exist 
elements $y_1, \ldots, y_r$ in $A$ such that $A$ is integral over $B = k[y_1, \ldots, y_r]$
and $N \cap  B = (y_1, \ldots, y_s)B$. It follows that $y_1, \ldots, y_r$ are algebraically
independent over $k$ and $A$  is a finitely generated $B$-module. Let 
$F$ denote the 
field $k(y_{s+1}, \ldots, y_r)$, 
and let $U$ denote the multiplicatively closed set $k[y_{s+1}, \ldots, y_r] \setminus (0)$.  
Then $U^{-1}B $ is the polynomial ring $F[y_1, \ldots, y_s]$, and  $U^{-1}A  := C$ 
is a finitely generated $U^{-1}B$-module. Moreover     $NC$ is  a prime  ideal of $C$ such 
that 
$$NC \cap U^{-1}B=(y_1,\ldots,y_s)U^{-1}B=(y_1,\ldots,y_s)F[y_1,\ldots,y_s]$$
 is a  maximal ideal   of $U^{-1}B$, and $(y_1,\ldots,y_s)C$ is primary for 
the maximal ideal of~$C$. Hence $NC$ is a maximal ideal of $C$
and $S = C_{NC}/QC_{NC}$ is a localization of the finitely generated $F$-algebra 
$D  := C/QC$ at the maximal ideal $NC/QC$. 

Therefore  $S$  is  
 a localization of an integral domain $D$ at a maximal ideal of $D$ and 
  $D$  is 
a  finitely generated algebra over an extension  field $F$ of $k$. 
\end{discussion}

\begin{theorem}     \label{esfingen2}    Let  $S$   be  a local integral domain of dimension $d$ that is essentially 
finitely generated over a field. Then every maximal ideal of the generic formal fiber ring $\Gff(S)$ has height $d-1$.
\end{theorem}

\begin{proof}   Using Discussion~\ref{discuss1}, we write $S=D_N$, where $N$ is a maximal ideal of a finitely generated algebra $D$ over a field $F$.
Let $\n=NS$ be the maximal ideal of $S$.
Choose $x_1, \ldots, x_d$ in $\n$ such that $x_1, \ldots, x_d$ are algebraically independent over $F$ and 
$(x_1, \ldots, x_d)S$ is $\n$-primary. Set $R = F[x_1, \ldots, x_d]_{(x_1, \ldots, x_d)}$, a localized polynomial ring over $F$,
and let $\m = (x_1, \ldots, x_d)R$. 

 To prove Theorem~\ref{esfingen2},  it suffices to show that the 
inclusion map $\phi: R \hookrightarrow S$ satisfies items 1 - 5 of Theorem~\ref{tgflem}.  By 
construction $\phi$ is an injective local homomorphism and $\m S$ is $\n$-primary.  Also $R/\m = F$ and $S/\n = D/N$ is a
field that is a finitely generated $F$-algebra and 
hence a finite algebraic extension field 
of $F$; see \cite[Theorem~5.2]{M}. Therefore item~1 holds.
Since $\dim S = d = \dim D$, 
the field of fractions of  $S$ has transcendence degree $d$
over the field $F$.  Therefore $S$ is
algebraic over $R$.  It follows that $R \hookrightarrow S$ is a TGF extension.  Thus item~2 holds.  
Since $R$ is a regular local ring,  $R$ is analytically irreducible and analytically normal.  
Since $S$ is essentially 
finitely generated over a field, $S$ is universally catenary.  Therefore items~3 and 4 hold. Since $R$ is a localized 
polynomial ring in $d$ variables,  Theorem~\ref{weierkxt} implies that every maximal ideal of $\Gff(R)$ has height $d-1$.
By  Theorem~\ref{tgflem}, every maximal ideal of $\Gff(S)$ has height $d-1$.  
\end{proof}

\section{Other results on generic formal fibers} \label{other}
 The main theorem of \cite{weier} includes results about 
the generic formal fiber ring of mixed polynomial-power series rings 
as in  Theorem~\ref{weieroth}.

\begin{theorem}  \label{weieroth}$\phantom{x}$  \cite[Theorem~24.1]{powerbook} Let   $m$ and $n$ be  positive integers, let $k$ be a field, and  let 
$X = \{x_1, \ldots, x_n\}$ and  $Y = \{y_1, \ldots, y_m\}$ be  sets of 
independent variables over $k$.
Then, for $R$ either the ring $ k[[X]]\,[Y]_{(X, Y)}$  or  
the ring $k[Y]_{(Y)}[[X]]$,   the dimension of the generic formal 
fiber ring $\Gff(R)$ is $n+m - 2$ and 
every prime
 ideal   $P$ maximal  in  $\Gff(R)$ has $\hgt P = n+m-2$. 
\end{theorem}

We use Theorem~\ref{tgflem} and Theorem~\ref{weieroth} 
to deduce Theorem~\ref{wogffgen}.

\begin{theorem}  \label{wogffgen}   Let $R$ be either $k[[X]]\,[Y]_{(X, Y)}$ or  $k[Y]_{(Y)}[[X]]
$, where  $m$ and $n$ are  positive integers and 
$X = \{x_1, \ldots, x_n\}$ and  $Y = \{y_1, \ldots, y_m\}$ are  sets of 
independent variables over  a field $k$.
Let  $\m$ denote the maximal ideal $(X,Y)R$  of $R$. Let $(S,\n)$ be a Noetherian local integral domain containing $R$ such that$~\!\!:$
\begin{enumerate}\item[$(1)$]
The injection
$\varphi: (R,\m) \hookrightarrow (S,\n)$ is a local map. 
\item[$(2)$]  $\m S$ is $\n$-primary,  and $S/\n$ is finite algebraic over $R/\m$. 
\item[$(3)$]   $R \hookrightarrow S$ is a TGF-extension and $\dim R = \dim S$.  
\item[$(4)$]  $S$ is  universally catenary.
\end{enumerate} Then every maximal ideal of
the generic formal fiber ring $\Gff(S)$  has height
$n+m-2$. Equivalently, if $P$ is a  prime ideal
of $\widehat S$ maximal with respect to $P\cap S=(0),$ then
$\hgt(P)=n+m-2$. 
\end{theorem}
\begin{proof} We check that the conditions 1--5 of Theorem~\ref{tgflem} are satisfied 
for $R$ and  $S$ and the injection $\varphi$. 
Since the completion of $R$ is $k[[X,Y]]$, $R$ is analytically normal, and so also analytically irreducible. Items 1--4 of Theorem~\ref{wogffgen}  
ensure that the rest of conditions 1--4 of Theorem~\ref{tgflem} hold. 
By Theorem~\ref{weieroth}, every maximal ideal of $\Gff(R)$ has height~$n+m-2$, and so condition 5 of Theorem~\ref{tgflem} holds. Thus by Theorem~\ref{tgflem},
every maximal ideal of $\Gff(S)$  has height
$n+m-2$. 
\end{proof}

\begin{remark}   \label{examples for mixed}    Let $k,X,Y,$ and $R$ be as in 
Theorem~\ref{wogffgen}.  Let $A$ be a finite integral extension domain of $R$ 
and let $S$ be the localization of $A$ at a maximal ideal.   
As observed in the proof of Theorem~\ref{wogffgen},  
$R$ is a local analytically normal integral domain. Since $S$ is a localization 
of a finitely generated $R$-algebra and $R$ is universally catenary, it follows 
that $S$ is universally catenary. 
We also have that conditions 1--3 of Theorem~\ref{wogffgen} hold. 
Thus  the extension $R \hookrightarrow S$ satisfies the hypotheses of Theorem~\ref{wogffgen}.
Hence   every maximal ideal of $\Gff(S)$  has height
$n+m-2$. 
\end{remark}  

Example~\ref{wogffgene} is an application of Theorem~\ref{wogffgen} and Remark~\ref{examples for mixed}. 

\begin{example} ~\label{wogffgene} Let $k,X,Y,$ and $R$ be as in Theorem~\ref{wogffgen}. 
Let $K$ denote the field of fractions of $R$,  and let $L$ be a finite 
algebraic extension field of $K$. Let $A$ be the 
integral closure of $R$ in $L$, and let $S$ be a localization of $A$  at a 
maximal ideal.   The ring $R$ is a Nagata ring by \cite[Prop.3.5]{Marot}. 
Therefore $A$ is a finite integral extension of $R$ and 
the conditions of Remark~\ref{examples for mixed} apply to show
that every maximal ideal of $\Gff(S)$  has height
$n+m-2$. 
\end{example}

\begin{remark} \label{essfgnfg2}   With  notation as in Example~\ref{wogffgene},   since the sets $X$ and $Y$ are nonempty, 
the field $K$ is a simple transcendental
extension of a subfield. It follows that 
the regular local ring $R$
is not Henselian,  see \cite[Satz~2.3.11, p. 60]{BKKN} and \cite{S}.  
 Hence there exists a finite algebraic field extension $L/K$  such that the integral closure $A$ of $R$ in $L$
has more than one maximal ideal.   It follows that  the localization $S$ of $A$ at any one  of these maximal ideals 
is not a finite $R$-module, and gives an
example $R \hookrightarrow S$ that  satisfies the hypotheses of  Theorem~\ref{tgflem}.  
\end{remark}

\end{document}